\documentclass[11pt,reqno]{amsart}

\newtheorem{proposition}{Proposition}[section]
\newtheorem{lemma}[proposition]{Lemma}
\newtheorem{corollary}[proposition]{Corollary}
\newtheorem{theorem}[proposition]{Theorem}

\theoremstyle{definition}
\newtheorem{definition}[proposition]{Definition}

\newtheorem{examples}[proposition]{Examples}
\newtheorem{remark}[proposition]{Remark}
\newtheorem{remarks}[proposition]{Remarks}

\newcommand{\thlabel}[1]{\label{th:#1}}
\newcommand{\thref}[1]{Theorem~\ref{th:#1}}
\newcommand{\selabel}[1]{\label{se:#1}}

\newcommand{\lelabel}[1]{\label{le:#1}}

\newcommand{\prlabel}[1]{\label{pr:#1}}
\newcommand{\prref}[1]{Proposition~\ref{pr:#1}}
\newcommand{\colabel}[1]{\label{co:#1}}
\newcommand{\coref}[1]{Corollary~\ref{co:#1}}

\newcommand{\delabel}[1]{\label{de:#1}}

\newcommand{\eqlabel}[1]{\label{eq:#1}}
\newcommand{\equref}[1]{(\ref{eq:#1})}

\newcommand{\Aut}{{\rm Aut}\,}

\newcommand{\im}{{\rm Im}\,}

\def\ZZ{{\mathbb Z}}

\newcommand{\Cc}{\mathcal{C}}

\def\*C{{}^*\hspace*{-1pt}{\Cc}}
\def\text#1{{\rm {\rm #1}}}

\def\cat{{{}_\Gamma\mathcal{C}}}
\def\Cat{{\mathcal{C}_\Gamma}}

\usepackage{amssymb}

\usepackage [UglyObsolete,tight,heads=LaTeX]{diagrams}
\usepackage{color,amssymb,graphicx}

\begin{document}

\title[Crossed product of groups. Applications]
{Crossed product of groups. Applications}

\author{A. L. Agore}\thanks{The authors where supported by
CNCSIS grant 24/28.09.07 of PN II "Groups, quantum groups, corings
and representation theory".}
\address[A.L.A., G.M.]{Faculty of Mathematics and Computer Science,
University of Bucharest, Str. Academiei 14, RO-010014 Bucharest 1,
Romania} \email{ana.agore@fmi.unibuc.ro}
\email{gigel.militaru@gmail.com or gigel.militaru@fmi.unibuc.ro}
\author{G. Militaru}\subjclass{20A05, 20E22, 20J15}

\keywords{The extension problem, crossed product}

\begin{abstract}{We survey the extensions of a group by a group using
crossed products instead of exact sequences of groups. The
approach has various advantages, one of them being that the
crossed product is a universal object. Several new applications
are given, a general Schreier type theorem is proved and a few
open problems are posed.}
\end{abstract}

\maketitle

\section*{Introduction}
In 1895, in a long paper on extensions of groups, O. L. H\"{o}lder
\cite{holder} launched one of the most interesting problems of
algebra: \emph{the extension problem}. Let $H$ and $G$ be two
groups. The extension problem consists of describing and
classifying all groups $E$ containing $H$ as a normal subgroup
such that $E/H \cong G$. The meaning of describing and,
especially, of classifying these structures is actually part of
the problem. The extension problem has been the starting point of
new subjects in mathematics such as cohomology of groups
(\cite{adem}), homological algebra (\cite{Wibel}), crossed
products of groups acting on algebras (\cite{pasman}), crossed
products of Hopf algebras acting on algebras (\cite{BCM}), crossed
products for von Neumann algebras \cite{nakagami} etc.

The first notable result regarding the extension problem was given
by O. L. H\"{o}lder himself (\thref{crossedciclice} below), who
uses generators and relations to describe all extensions of a
finite cyclic group by another finite cyclic group. The second
major contribution regarding this problem was given by O. Schreier
in 1926: he classified all extensions in case $H$ is abelian and
the morphisms are defined such that they stabilize $H$ and $G$.
The set of equivalence classes of extensions (via the equivalence
relation that stabilizes $H$ and $G$) is in a 1-1 correspondence
with the cohomology group $H^ 2 (G, H)$. The third and most
important contribution to the extension problem was given in 1947
by S. Eilenberg and S. MacLane in two fundamental papers
\cite{EM}. In general, no solution is known for the general
classification problem (cf. \cite[page 155]{rotman}) although J.
Baez recently stated in \cite{baez} that the above extensions are
classified by weak $2$-functors $G \to AUT (H)$.

The classification part of the extension problem requires first of
all a definition of the morphisms between such objects. In other
words, we have to define the category of extensions of a group $H$
by a group $G$. Surprisingly, these morphisms have so far received
only one definition, which is not the only one possible. More
precisely: a morphism between two extensions $(E, i, p)$ and $(E',
i', p')$ of $H$ by $G$ is a morphism of groups $\gamma :
E\rightarrow E'$ for which the following diagram is commutative:
$$\begin{diagram}
1 \rTo & H & \rTo^{i} & E & \rTo^{p} & G & \rTo & 1\\
& \dTo_{Id_{H}} & & \dTo_{\gamma} & & \dTo_{Id_{G}} & \\
1 \rTo & H & \rTo^{i'} & E' & \rTo^{p'} & G & \rTo & 1
\end{diagram}$$
The category $\mathcal{E}_1 (H, G)$ obtained this way is a
groupoid (i.e. any morphism is an isomorphism) and all
classification results proved until now on the extension problem
were related only to this category. Is this the only way to define
a morphism between two extensions? Of course not, they can be
defined in at least two other ways, one being the following, which
is the most natural from the point of view of category theory: a
morphism between two extensions $(E, i, p)$ and $(E', i', p')$ of
$H$ be $G$ is a triple $(u, \gamma, v)$, where $u: H \to H$,
$\gamma : E \to E'$, $v: G \to G$ are morphisms of groups such
that the following diagram is commutative:
$$\begin{diagram}
1 \rTo & H & \rTo^{i} & E & \rTo^{p} & G & \rTo & 1\\
& \dTo_{u} & & \dTo_{\gamma} & & \dTo_{v} & \\
1 \rTo & H & \rTo^{i'} & E' & \rTo^{p'} & G & \rTo & 1
\end{diagram}$$
This way we obtain a new and very different category
$\mathcal{E}_2 (H, G)$, which is not a groupoid any more and is
connected to $\mathcal{E}_1 (H, G)$ via a faithful functor
$$
F : \mathcal{E}_1 (H, G) \to \mathcal{E}_2 (H, G), \quad F \bigl(
(E, i, p) \bigl) = (E, i, p), \quad F (\gamma) = (Id_H, \gamma,
Id_G)
$$
In this paper we will approach the classification part of the
extension problem by defining the morphisms as above. The main
result (\thref{teoprinci}) provides this classification. In order
to do this, we had to replace the extensions $(E, i, p)$ with the
equivalent concept of crossed systems $(H, G, \alpha, f)$. The
transition from extensions to crossed systems is quite natural if
we ask ourselves the following elementary question:

\emph{Let $H$ be a group and $E$ be a set such that $H \subseteq
E$. What are all the group structures $(E, \cdot)$ that can be
defined on the set $E$ such that $H \unlhd E$ is a normal subgroup
of $E$?}

Let $(E, \cdot)$ be such a group structure and $G := E/H$ be its
quotient group. Then, as a set, $E \cong H \times G$ and hence the
problem can be restated as follows: let $H$ and $G$ be two groups;
what are all the group structures that can be defined on the set
$H\times G$ with the property that $H\cong H\times \{ 1\} \unlhd H
\times G$? The set of these structures is in a one to one
correspondence with the set of all crossed systems $(H, G, \alpha,
f)$, where $\alpha : G \rightarrow \Aut (H)$ is a weak action and
$f : G \times G \rightarrow H$ is an $\alpha$-cocycle. Fixing the
groups $H$ and $G$ and denoting by ${\rm Crossed} \, (H, G)$ the
set of all normalized crossed systems we shall identify three
categories having the same class of objects, namely the set ${\rm
Crossed} \, (H, G)$. Thus, the classification part of the
extension problem can be restated in a much more precise,
categorical way: describe the skeleton of these categories. The
main result of this paper describes the skeleton of the category
$\mathcal{E}_2 (H, G)$ (\thref{teoprinci}). The crossed product
satisfies two universal properties: it is an initial object in a
category but also a final object in another category which is not
a dual of the first one (\thref{2.1.8}). This result has several
applications; in particular, the set of all (iso)morphisms between
two crossed products $H \#_{\alpha}^{f} \, G$ and $H
\#_{\alpha'}^{f'}\, G$ is explicitly described. It is in a one to
one correspondence with the set of all quadruples $(u, r, v, s)$,
where $u : H \to H$, $r: G\to H$, $ v: G \to G$ are maps, and $s:
H \to G$ is a morphism of groups satisfying certain compatibility
conditions (\coref{morfisme1}, \coref{morfisme2}). On the route
other interesting results are derived and a few open questions are
posed.

\section{Preliminaries}\selabel{1}

\subsection{Definitions and notation}\selabel{1.1}
Let us fix the notations that will be used throughout the paper.
$|A|$ denotes the number of elements of a finite set $A$ and $C_n$
will be a cyclic group of order $n$ generated by $a$: $C_n = \{1,
a, a^2, \cdots , a^{n-1} \}$. Let $H$ and $G$ be two groups. $\Aut
(H)$ denotes the group of automorphisms of a group $H$ and $Z(H)$
the center of $H$. A map $f : G \times G \rightarrow H$ is called
\emph{symmetric} if $f (g_1, g_2) = f (g_2, g_1)$ for any $g_1$,
$g_2 \in G$. For a map $\alpha : G \rightarrow \Aut (H)$ we shall
use the notation
$$
\alpha (g) (h) = g\triangleright h
$$
for all $g\in G$ and $h\in H$. As $\alpha (g) \in \Aut (H)$ we
have that
\begin{equation}\eqlabel{8cross}
g \triangleright 1 = 1, \qquad g \triangleright (h_1h_2) = (g
\triangleright h_1) (g \triangleright h_2), \qquad g
\triangleright h^{-1} = (g \triangleright h)^{-1}
\end{equation}
for any $g\in G$ and $h$, $h_1$, $h_2\in H$. The map $\alpha$ is
called \emph{trivial} if $g\triangleright h = h$ for all $g\in G$
and $h\in H$. If $\alpha$ is a morphism of groups we denote by $H
\ltimes_{\alpha} G$ the semidirect product of $H$ and $G$: $H
\ltimes_{\alpha} G = H\times G$ as a set with the multiplication
given by
$$
(h_1,\, g_1)\cdot (h_2, \, g_2) : = \bigl( h_1 (g_1\triangleright
h_2), \, g_1 g_2 \bigl)
$$
for all $h_1$, $h_2 \in H$, $g_1$, $g_2 \in G$.

An \emph{extension of $H$ by $G$} is a triple $(E, i, \pi)$, where
$E$ is a group, $i : H \to E$ and $\pi : E \to G$ are morphisms of
groups such that the sequence
$$
\begin{diagram}
1 \rTo & H & \rTo^{i} & E & \rTo^{\pi} & G & \rTo & 1
\end{diagram}
$$
is exact.

\subsection{Crossed product of groups}\selabel{1.2} We recall now
a fundamental construction at the level of groups. It has served
as a model for later generalizations at the level of, e.g., groups
acting on rings \cite{pasman}, Hopf algebras acting on algebras
\cite{BCM}, von Neumann algebras \cite{nakagami}, quantum
groupoids \cite{BB}.

\begin{definition}\delabel{crossedsystem}
A \textit{crossed system} of groups is a quadruple $\Gamma = (H,
G, \alpha, f)$, where $H$ and $G$ are two groups, $\alpha : G
\rightarrow \Aut (H)$ and $f : G \times G \rightarrow H$ are two
maps such that the following compatibility conditions hold:
\begin{eqnarray}
g_1 \triangleright (g_2\triangleright h) &=& f(g_1, g_2) \,
\bigl((g_1g_2)\triangleright  h \bigl)\, f(g_1, g_2)^{-1}
\eqlabel{WA} \\
f(g_1,\, g_2)\, f(g_1 g_2, \, g_3) &=&  \bigl(g_1 \triangleright
f(g_2, \, g_3) \bigl) \, f(g_1, \, g_2g_3) \eqlabel{CC}
\end{eqnarray}
for all $g_1$, $g_2$, $g_3 \in G$ and $h\in H$. The crossed system
$\Gamma = (H, G, \alpha, f)$ is called \textit{normalized} if
$f(1, 1) = 1$. The map $\alpha : G \rightarrow \Aut (H)$ is called
a \textit{weak action} and $f : G \times G \rightarrow H$ is
called an $\alpha$-\textit{cocycle}.
\end{definition}

We denote by ${\rm Crossed} \, (H, G)$ the set of all normalized
crossed systems:
$${\rm Crossed} \, (H, G) = \{ (\alpha, f) \, | \, (H,
G, \alpha, f) \, {\rm is \, a \, normalized \, crossed \, system}
\}
$$
We note that if ${\rm Im} (f) \subseteq Z(H)$ the condition
\equref{WA} is equivalent to the fact that $\alpha$ is an action:
$(g_1 g_2)\triangleright h = g_1 \triangleright (g_2\triangleright
h)$, for all $g_1$, $g_2 \in G$ and $h\in H$. First we give some
useful formulas for a crossed system.

\begin{lemma}\lelabel{forcross}
Let $(H, G, \alpha, f)$ be a crossed system. Then
\begin{eqnarray}
 f(g, 1) &=& g \triangleright f(1,1) \eqlabel{5cross} \\
 1\triangleright h &=& f(1,1) \, h \, f(1,1)^{-1} \eqlabel{6cross}
 \\
f(1, g)  &=& f(1,1) \eqlabel{7cross}
\end{eqnarray}
for any $g\in G$ and $h\in H$. In particular, if $(H, G, \alpha,
f)$ is a normalized crossed system then
\begin{equation}\eqlabel{norm2}
f (1, g) = f (g, 1) = 1 \qquad {\rm and} \qquad 1 \triangleright h
= h
\end{equation}
for any $g \in G$ and $h\in H$.
\end{lemma}

\begin{proof}
The condition \equref{CC} for $g_2 = g_3 = 1$ and $g_1 = g$ gives
\equref{5cross}. Now if we set $g_1 = g_2 =1$ in \equref{WA} and
take into account that $\alpha (1)$ is surjective we obtain
\equref{6cross}. If we set $g_1 = g_2 = 1$ and $g_3 = g$ in
\equref{CC} and take into account \equref{6cross} we obtain
\equref{7cross}.
\end{proof}

Let $H$ and $G$ be groups, $\alpha : G \rightarrow \Aut (H)$ and
$f : G \times G \rightarrow H$ two maps. Let $H \#_{\alpha}^{f} \,
G : = H\times G$ as a set with a binary operation defined by the
formula:
\begin{equation}\eqlabel{4}
(h_1,\, g_1)\cdot (h_2, \, g_2) : = \bigl( h_1 (g_1 \triangleright
h_2) f(g_1, \, g_2), \, g_1g_2 \bigl)
\end{equation}
for all $h_1$, $h_2 \in H$, $g_1$, $g_2 \in G$.

The following theorem gives the construction of the crossed
product of groups. It is unfortunately difficult to refer to a
place that would contain the proof of this version of theorem (see
\cite{adem}, \cite{alperin}, \cite{brown}, \cite{bechtell},
\cite{grillet}, \cite{hump}, \cite{rotman}, \cite{Wibel}).
Therefore, for convenience purposes, we present a short proof
below:

\begin{theorem}\thlabel{3}
Let $H$ and $G$ be groups $\alpha : G \rightarrow \Aut (H)$ and $f
: G \times G \rightarrow H$ two maps. The following statements are
equivalent:
\begin{enumerate}
\item The multiplication on $H \#_{\alpha}^{f} \, G$ given by
\equref{4} is associative.

\item $(H, G, \alpha, f)$ is a crossed system.
\end{enumerate}
In this case $\bigl ( H \#_{\alpha}^{f} \, G, \, \cdot \bigl)$ is
a group with the unit $1_{H \#_{\alpha}^{f} \, G} = \bigl( f(1, 1)
^{-1}, \, 1\bigl)$ called the \textit{crossed product of $H$ and
$G$} associated to the crossed system $(H, G, \alpha, f)$.
\end{theorem}

\begin{proof} For $h_1$, $h_2$, $h_3 \in H$ and
$g_1$, $g_2$, $g_3 \in G$ we have
$$
[(h_1, g_1) \cdot (h_2, g_2)]\cdot (h_3, g_3) = ( h_1
(g_1 \triangleright h_2) \underline{f(g_1, g_2)
((g_1g_2)\triangleright h_3) f(g_1g_2, g_3)}, \,
 g_1g_2g_3)
$$
and
$$
(h_1, g_1)\cdot[(h_2, g_2)\cdot(h_3, g_3)] = ( h_1 (g_1
\triangleright h_2) \underline{\bigl(g_1 \triangleright (g_2
\triangleright h_3 ) \bigl) \bigl( g_1 \triangleright f(g_2, g_3)
\bigl) f(g_1, g_2g_3)}, \, g_1g_2g_3)
$$
Hence, the multiplication given by \equref{4} is associative if
and only if
\begin{equation}\eqlabel{WACC}
f(g_1, g_2) ((g_1g_2)\triangleright h_3) f(g_1g_2, g_3) =
\bigl(g_1 \triangleright (g_2 \triangleright h_3 ) \bigl) \bigl(
g_1 \triangleright f(g_2, g_3) \bigl) f(g_1, g_2g_3)
\end{equation}
for all $g_1$, $g_2$, $g_3 \in G$ and $h_3 \in H$. We shall prove
now that \equref{WACC} holds if and only if \equref{WA} and
\equref{CC} holds.

Assume first that \equref{WA} and \equref{CC} holds. Then
\begin{eqnarray*}
\underline{ f(g_1, g_2) \, \bigl( (g_1g_2)\triangleright h_3
\bigl)} f(g_1g_2, g_3) &\stackrel{\equref{WA}} {=}& \bigl( g_1
\triangleright (g_2
\triangleright h_3 ) \bigl) f (g_1, g_2) f(g_1g_2, g_3) \\
&\stackrel{\equref{CC}} {=}& \bigl(g_1 \triangleright (g_2
\triangleright h_3 ) \bigl) \bigl( g_1 \triangleright f(g_2, g_3)
\bigl) f(g_1, g_2g_3)
\end{eqnarray*}
i.e. \equref{WACC} holds. Conversely, assume that \equref{WACC}
holds. Using \equref{8cross} after we specialize $h_3 = 1$ in
\equref{WACC} we obtain \equref{CC}. Now,
\begin{eqnarray*}
\Bigl( g_1 \triangleright (g_2\triangleright h) \Bigl) f(g_1, g_2)
&\stackrel{\equref{WACC}} {=}& f(g_1, g_2) \bigl((g_1 g_2)
\triangleright h\bigl) f(g_1g_2, g_3) \\
&& [(g_1 \triangleright f(g_2, g_3)) f(g_1, g_2g_3)]^{-1} f(g_1, g_2) \\
&\stackrel{\equref{CC}} {=}& f(g_1, g_2) \bigl((g_1 g_2)
\triangleright h \bigl) f(g_1g_2, g_3) f(g_1g_2, g_3)^{-1}
\\
&&f(g_1, g_2)^{-1} f(g_1, g_2)\\
&=& f(g_1, g_2) \bigl((g_1 g_2) \triangleright h \bigl)
\end{eqnarray*}
i.e. \equref{WA} holds; hence the first part of the theorem is
proved.

We assume now that $(H, G, \alpha, f)$ is a crossed system and we
prove that $\bigl ( H \#_{\alpha}^{f} \, G, \, \cdot \bigl)$ is a
group. For $h\in H$ and $g\in G$ we have
\begin{eqnarray*}
(h, g)\cdot(f(1, 1)^{-1}, 1) &=& \Bigl ( h \bigl(g\triangleright
(f(1,1)^{-1})\bigl ) f(g, 1), \, g \Bigl) \\
&\stackrel{\equref{5cross}} =& \Bigl( h (g\triangleright
(f(1,1)^{-1})\,
(g\triangleright f(1,1)), \, g \Bigl) \\
& \stackrel{\equref{8cross}}  = & \Bigl
(h (g \triangleright (f(1,1)^{-1} f(1,1)) ) , \, g\Bigl) \\
&=& (h (g\triangleright 1 ), g) = (h, g)
\end{eqnarray*}
and
\begin{eqnarray*}
(f(1, 1)^{-1}, 1) \cdot (h, g) &=& \bigl( f(1, 1)^{-1} (1
\triangleright
h ) f(1, g), \, g\bigl ) \\
&\stackrel{\equref{6cross}} =& \bigl( f(1, 1)^{-1} f(1, 1) h f(1,
1)^{-1} f(1,g), \, g \bigl ) \\
&\stackrel{\equref{7cross}} =& \bigl (h f(1, 1)^{-1} f(1, 1), \, g
\bigl) = (h, g)
\end{eqnarray*}
i.e. $(f(1, 1)^{-1}, 1)$ is the unit of $\bigl ( H \#_{\alpha}^{f}
\, G, \, \cdot \bigl)$. Let now $(h, g) \in H \#_{\alpha}^{f} \,
G$. Then it is easy to see that
$$
(h, g)^{-1} = \bigl(f(1, 1) ^{-1} f(g^{-1}, g) ^{-1} (g^{-1}
\triangleright h^{-1}), \, g^{-1} \bigl)
$$
is a left inverse of $(h, g)$. Thus $H \#_{\alpha}^{f} \, G$ is a
monoid  and any element of it has a left inverse. Then $H
\#_{\alpha}^{f} \, G$ is a group and we are done.
\end{proof}

Let $\Gamma = (H, G, \alpha, f)$ be a normalized crossed system.
Then in the crossed product  $H \#_{\alpha}^{f} \, G$ we have:
\begin{equation}\eqlabel{2.1.gen}
(h, 1) \cdot (1, g) = (h, g)
\end{equation}
for any $h\in H$ and $g\in G$. Thus $(H\times \{1\}) \cup (\{1\}
\times G)$ is a set of generators of the group $H \#_{\alpha}^{f}
\, G$. An extension of $H$ by $G$ is associated to any crossed
system as follows:

\begin{corollary}\colabel{zz}
Let $(H, G, \alpha, f)$ be a crossed system. Then
\begin{equation}\eqlabel{sir5}
\begin{diagram}
1 \rTo & H & \rTo^{i_H} & H \#_{\alpha}^{f} \, G & \rTo^{\pi_G} & G
& \rTo & 1
\end{diagram}
\end{equation}
where $i_H (h):= \bigl( hf(1, 1) ^{-1}, 1\bigl)$ and $\pi_G (h,
g): = g$ for all $h \in H$ and $g\in G$ is an exact sequence of
groups, i.e. $(H \#_{\alpha}^{f} \, G, i_H, \pi_G)$ is an
extension of $H$ by $G$.
\end{corollary}

\begin{examples}
1. Let $H$ and $G$ be two groups and $\alpha$, $f$ be the trivial
maps: i.e. $\alpha (g)(h) = h$ and $f (g_1, g_2) = 1$ for all $g$,
$g_1$, $g_2 \in G$ and $h \in H$. Then $\Gamma = (H, G, \alpha,
f)$ is a crossed system called the \textit{trivial crossed
system}. The crossed product $H \#_{\alpha}^{f} \, G = H\times G$,
the direct product of $H$ and $G$.
\newline 2. Let $H$ and $G$ be two groups and $f: G\times G \rightarrow H$ the
trivial map. Then $(H, G, \alpha, f)$ is a crossed system if and
only if $\alpha : G \rightarrow \Aut (H)$ is a morphism of groups.
In this case the crossed product $H \#_{\alpha}^{f} \, G = H
\ltimes_{\alpha} G$, the semidirect product of $H$ and $G$.
\newline 3. Let $H$ and $G$ be two groups and $\alpha : G
\rightarrow \Aut (H)$ the trivial action. Then $(H, G, \alpha, f)$
is a crossed system if and only if $\im (f) \subseteq Z(H)$ and
\begin{equation}\eqlabel{CCpropiu}
f(g_1, \, g_2) f(g_1g_2, \, g_3) = f(g_2, \, g_3) f(g_1, \,
g_2g_3)
\end{equation}
for all $g_1$, $g_2$, $g_3 \in G$, that is $f : G\times G
\rightarrow Z(H)$ is a $2$-cocycle as they appear in the abelian
cohomology of groups (\cite{adem}, \cite{rotman}, \cite{Wibel}).
The crossed product $H \#_{\alpha}^{f} \, G$ associated to this
crossed system will be denoted by $H\times^{f} \, G$ and we shall
call it the \textit{twisted product}\footnote{We borrowed the
terminology from groups acting on $k$-algebras} of $H$ and $G$
associated to the $2$-cocycle $f: G\times G \rightarrow Z(H)$.
Explicitly, the multiplication of a twisted product of groups
$H\times^{f} \, G$ is given by the formula:
\begin{equation}\eqlabel{tw4}
(h_1,\, g_1)\cdot (h_2, \, g_2) : = \bigl( h_1 h_2 f(g_1, \, g_2),
\, g_1g_2 \bigl)
\end{equation}
for all $h_1$, $h_2 \in H$, $g_1$, $g_2 \in G$.
\end{examples}

The next theorem shows that any extension $(E, i, \pi)$ of $H$ by
$G$ is equivalent to a crossed product extension  $(H
\#_{\alpha}^{f} \, G, i_H, \pi_G)$. It can be also viewed as a
reconstruction theorem of a group from a normal subgroup and the
quotient.

\begin{theorem}\thlabel{4}
Let $(E, i, \pi)$ be an extension of $H$ by $G$. Then there exists
$(H, G, \alpha, f)$ a normalized crossed system and an isomorphism
of groups $\theta: H \#_{\alpha}^{f} \, G \rightarrow E$ such that
the following diagram
$$
\begin{diagram}
1 \rTo & H & \rTo^{i_H} & H \#_{\alpha}^{f} \,
G & \rTo^{\pi_G} & G & \rTo & 1\\
& \dTo_{Id_H} & & \dTo_{\theta} & & \dTo_{Id_G} & \\
1 \rTo & H & \rTo^{i} & E & \rTo^{\pi} & G & \rTo & 1
\end{diagram}
$$
is commutative.
\end{theorem}

\begin{proof}
For full details we refer to \cite{rotman}, \cite{Wibel}. We shall
identify $ H \cong i (H) \unlhd E$. The crossed system is
constructed as follows: let $s : G \rightarrow E$ be a section of
$\pi : E \rightarrow G$ such that $s(1) = 1$ and define $\alpha$
and $f$ by the formulas:
\begin{equation}\eqlabel{8}
\alpha : G \rightarrow \Aut (H), \qquad \alpha (g)(h):= s(g)h
s(g)^{-1}
\end{equation}
\begin{equation}\eqlabel{9}
f : G \times G \rightarrow H, \qquad f(g_1, g_2) :=
s(g_1)s(g_2)s(g_1g_2)^{-1}
\end{equation}
for all $g$, $g_1$, $g_2\in G$ and $h\in H$. Then $(H, G, \alpha,
f)$ is a normalized crossed system and
$$
\theta: H \#_{\alpha}^{f} \, G \rightarrow E, \qquad \theta (h,
g):= i(h) s(g)
$$
is an isomorphism of groups and the diagram is commutative: $\pi
\bigl (\theta (h, g) \bigl) = \pi (i(h)) \pi (s(g)) = g = Id_G
\bigl(\pi_G (h, g) \bigl)$, for all $h\in H$ and $g\in G$.
\end{proof}

The next corollary shows that any crossed product of groups is
isomorphic to a crossed product of a normalized crossed system.

\begin{corollary}\colabel{zzz}
Let $(H, G, \alpha, f)$ be a crossed system. Then there exists
$(H, G, \alpha', f')$ a normalized crossed system such that $H
\#_{\alpha}^{f} \, G \cong H \#_{\alpha '}^{f'} \, G$ (isomorphism
of groups).
\end{corollary}

\begin{proof} It follows from the exact sequence \equref{sir5}
that $H \cong H \times \{1\} \unlhd H \#_{\alpha}^{f}G$ and
$H\#_{\alpha}^{f} G /H\times \{1\} \cong G$. Using \thref{4} we
obtain a normalized crossed system $(H, G, \alpha', f')$ such that
$H \#_{\alpha}^{f} \, G \cong H \#_{\alpha '}^{f'} \, G$.
\end{proof}

The following result is a better formulation of the Schreier
theorem \cite[Theorem 12.4]{grillet}.

\begin{corollary} Let $H$ and $G$ be two groups. The existence of the following data is equivalent:
\begin{enumerate}
\item An extension of $H$ by $G$. \item A normalized crossed
system $(H, G, \alpha, f)$. \item A crossed system $(H, G, \alpha,
f)$.
\end{enumerate}
\end{corollary}

\begin{proof}
It follows from \coref{zz}, \thref{4} and \coref{zzz}.
\end{proof}

Thus the extension problem of H\"{o}lder can be restated in a
computational manner as follows:

\emph{\textbf{Problem 1}: Let $H$ and $G$ be two fixed groups.
Describe all normalized crossed systems $(H, G, \alpha, f)$ and
classify up to isomorphism all crossed products $H \#_{\alpha}^{f}
\, G$.}

The description of all extensions of a group by a group (or,
equivalently in view of the corollary above, of all normalized
crossed systems that can be constructed for two fixed groups) has
been a central problem of group theory during the last century
(see for example \cite{adem}, \cite{brown}, \cite{hump} ). For the
second part of the problem 1 (namely the classification) no
solution is known in general. The first important result in the
literature for the first part of problem 1 was proved by
H\"{o}lder himself \cite[Theorem 12.9]{grillet}. It describes the
crossed product of two finite cyclic groups (we refer to \cite{AF}
for related results):

\begin{theorem}\thlabel{crossedciclice}
A finite group $E$ is isomorphic to a crossed product $C_n
\#_{\alpha}^{f} \, C_m$ if and only if $E$ is the group generated
by two generators $a$ and $b$ subject to the relations
$$a^n = 1, \qquad b^m = a^i, \qquad b^{-1} a b = a^j$$
where $i$, $j \in \{0, 1, \cdots , n-1 \}$ such that
$$i(j-1) \equiv 0 ({\rm mod} \, n), \qquad j^m \equiv 1 ({\rm mod}
\, n)$$
\end{theorem}

The following argument indicates the crucial importance of the
previous problem: we shall prove that any finite group is
isomorphic to a finite product of normalized crossed products of
finite simple groups. Thus we can survey all finite groups as if
we are able to compute various crossed systems starting with
finite simple groups or crossed products of them. The first part
of the following theorem can be found in \cite[pages
283-284]{rotman2} using the equivalently language of extensions of
a group by a group. Here we present a different proof.

\begin{theorem}\thlabel{2.2.9}
Any finite group is isomorphic to an iteration of normalized
crossed products of finite simple groups.
\newline Any abelian finite group is isomorphic to an iteration of normalized
twisted products of various $\ZZ_{p_i}$, where $p_i$ are prime
numbers.
\end{theorem}

\begin{proof}
Let $E$ be a finite group of order $n$. We will prove by induction
on $n$. If $n=2$ then $E\cong \ZZ_2$ and we are done. Assume that
$n > 2$. If $E$ is simple there is nothing to prove. Assume that
$E$ has a proper normal subgroup $\{1\}\neq H \lhd E$. It follows
from \thref{4} that there exists a normalized crossed system $ (H,
E/H, \alpha, f)$ such that $E\cong H \#_{\alpha}^{f} \, E/H$. As
$H$ and $E/H$ have order $< n$ we apply the induction. The abelian
case follows similarly if we apply \thref{4} and the fact that an
abelian simple group is isomorphic to $\ZZ_p$, for a prime number
$p$: in this case \equref{8} shows that in the abelian case the
action $\alpha$ that arises is trivial, i.e. the crossed product
between $H$ and $E/H$ is reduced to the twisted product $E\cong H
\times^{f} \, E/H$.
\end{proof}

The examples below show how the \thref{2.2.9} is applied.

\begin{examples}
1. Let $f: C_2 \times C_2 \rightarrow C_2$ given by $f(1,1) =
f(1,a) = f(a,1) = 1$, $f(a,a) = a$. Then $f$ is a $2$-cocycle and
$C_4\cong C_2 \times^{f} \, C_2 $, the twisted product of $C_2$
and $C_2$.

2. Let
$$
\alpha : C_2 \rightarrow \Aut (C_4), \qquad \alpha (1)= {\rm
Id}_{C_4}, \quad \alpha (a) (x) = x^{-1}
$$
for all $x\in C_4$ and
$$f: C_2 \times C_2 \rightarrow C_4, \qquad
f(1,1) = f(1,a) = f(a,1) = 1, \quad f(a,a) = b
$$
where $b$ is a generator of $C_4$. Then $(C_4, C_2, \alpha, f)$ is
a normalized crossed system and $C_4 \#_{\alpha}^{f} \, C_2 \cong
Q$, the quaternion group $Q$ of order $8$ . Thus the quaternion
group $Q$ can be presented as
$$
Q \cong (C_2 \times^{f_1} \, C_2) \#_{\alpha}^{f_2} \, C_2
$$
for a $2$-cocycle $f_1$, a weak action $\alpha$ and a
$\alpha$-cocycle $f_2$.
\end{examples}

\thref{2.2.9} leads naturally to the question of associativity of
crossed products.

\emph{\textbf{Problem 2}: Let $H$, $G$ and $K$ be three groups and
$(H, G, \alpha, f)$ and $(H \#_{\alpha}^{f} \, G, K, \beta, g )$
two crossed systems. Under what circumstances do two crossed
systems $(G, K, \beta', g')$ and $(H, G \#_{\beta'}^{g'} \, K,
\alpha', f')$ and an isomorphism of groups
$$
\bigl( H \#_{\alpha}^{f} \, G \bigl) \#_{\beta}^{g} \, K \cong H
\#_{\alpha '}^{f' } \, \bigl (G \#_{\beta'}^{g'} \, K \bigl) \,
exist ?
$$}
\subsection{Commutativity of crossed product}
The results presented here are the counterpart at the level of
groups of some theorems proved recently in \cite{os} for crossed
products of group actions on rings.

We shall fix $\Gamma = (H, G, \alpha, f)$, a normalized crossed
system of groups. We define the group of invariants of the weak
action $\alpha$ as follows
$$
H^G := \{h \in H \, | \, g \rhd h = h, \, \forall g\in G \}
$$
Then $H^G$ is a subgroup of $H$ called the \emph{subgroup of
invariants} of the crossed system and hence we have the following
extension of groups that is associated to a crossed system $(H, G,
\alpha, f)$:
$$
H^G \leq H \cong H \times \{1 \} \unlhd H \#_{\alpha}^{f} \, G
$$

\emph{\textbf{Problem 3}: Let $(H, G, \alpha, f)$ be a normalized
crossed system. Give a necessary and sufficient condition for the
category of representations of H $\#_{\alpha}^{f} \, G$ to be
equivalently to the category of representations of $H^G$. }

We compute now the center of a crossed product:

\begin{proposition}\prlabel{centrudec}
Let $(H, G, \alpha, f)$ be a normalized crossed system. Then $ (h,
g) \in Z ( H \#_{\alpha}^{f} \, G )$ if and only if the following
three conditions hold
\begin{eqnarray}
g\rhd h' &=& h^{-1} h' h, \quad g \in  Z(G)   \eqlabel{centru2} \\
(g' \rhd h) f(g', g) &=& h f (g, g') \eqlabel{centru3}
\end{eqnarray}
for any $h' \in H$ and $g'\in G$.
\end{proposition}

\begin{proof}
It follows from \equref{2.1.gen} that $(H\times \{1\}) \cup (\{1\}
\times G)$ is a set of generators for $H \#_{\alpha}^{f} \, G$.
Thus, $ (h, g) \in Z ( H \#_{\alpha}^{f} \, G )$ if and only if $
(h, g)$ commutes with $(h', 1)$ and $ (1, g')$ for all $h' \in H$
and $g'\in G$. A direct computation shows that $(h, g) \cdot (1,
g') = (1, g') \cdot (h, g)$ for all $g'\in G$ if and only if $g\in
Z(G)$ and \equref{centru3} holds. Similarly, using \equref{norm2}
we can show that $(h, g) \cdot (h', 1) = (h', 1) \cdot (h, g)$ for
all $h'\in H$ if and only if the first condition in
\equref{centru2} holds.
\end{proof}

The next corollary gives the center of a twisted product:

\begin{corollary} \colabel{centrutw}
Let $H$ and $G$ be two groups, $f : G\times G \rightarrow Z(H)$ a
normalized $2$-cocycle and $H\times^{f} \, G$ the twisted product
of $H$ and $G$. Then
$$
Z ( H\times^{f} \, G ) = \{ (h, g) \in Z(H) \times Z (G) \, | \,
f( -, g) = f (g, -) \}
$$
In particular, a twisted product $H\times^{f} \, G$ is an abelian
group if and only if $H$ and $G$ are abelian groups and $f$ is a
symmetric 2-cocycle.
\end{corollary}

Using the subgroup $H^G$ of invariants we give a description of
the center of a crossed product having a symmetric cocycle:

\begin{corollary} Let $(H, G, \alpha, f)$ be a normalized crossed
system such that $f$ is a symmetric $\alpha$-cocycle. Then
$$
Z ( H \#_{\alpha}^{f} \, G ) = \{ (h, g) \in H^G \times Z (G) \, |
\,  g \rhd h' = h^{-1} h' h, \, \forall \, h'\in H \}
$$
\end{corollary}

The next result gives a necessary and sufficient condition for a
crossed product to be an abelian group.

\begin{corollary}\colabel{crocicl}
Let $(H, G, \alpha, f)$ be a normalized crossed system. Then $H
\#_{\alpha}^{f} \, G$ is an abelian group if and only if $H$ and
$G$ are abelian groups, $\alpha$ is the trivial action and $f$ is
a symmetric 2-cocycle.
\end{corollary}

\begin{proof} Assume that $H \#_{\alpha}^{f} \, G$ is an abelian
group. Then $H$ and $G$ are abelian groups. Using \equref{centru2}
of \prref{centrudec} we obtain that $\alpha$ is the trivial action
and hence \equref{centru3} shows that $f$ is symmetric. The
converse follows from \coref{centrutw}.
\end{proof}

\begin{remark}
Let $(H, G, \alpha, f)$ be a normalized crossed system. Then the
centralizer of $H \cong H\times \{1 \}$ in $H \#_{\alpha}^{f} \,
G$ is given by
$$
C_{H \#_{\alpha}^{f} \, G} (H) =  \{ (h,g) \, | \, g \rhd x =
h^{-1} x h, \forall x \in H \}.
$$
If $H$ is abelian then $\alpha$ is a morphism of groups and $C_{H
\#_{\alpha}^{f} \, G} (H) = H \times {\rm Ker} (\alpha) $.
Moreover, if $G$ is also abelian and $f$ is a symmetric
$\alpha$-cocycle then $C_{H \#_{\alpha}^{f} \, G} (H)$ is an
abelian group. Indeed, for $(h, g)$, $(h', g') \in C_{H
\#_{\alpha}^{f} \, G} (H)$ we have:
\begin{eqnarray*}
(h, g) \cdot (h', g') &=& \bigl( h (g\rhd h') f (g, g'), gg'
\bigl) \\
&=& \bigl(h h^{-1} h' h f (g, g'), gg' \bigl)\\
&=& \bigl( h' h f (g, g'), gg' \bigl) = (h', g') \cdot (h, g)
\end{eqnarray*}
\end{remark}

\section{Crossed product as initial and final object. Applications}\selabel{2}
In this section $\Gamma = (H, G, \alpha, f)$ will be a normalized
crossed system of groups. We shall prove that the crossed product
$H \#_{\alpha}^{f} \, G$ is determined by a universal property in
two ways: it can be viewed at the same time as an initial object
in a certain category $\cat$ and as a final object into another
category $\Cat$. Define the category $\cat$ as follows: the
objects in $\cat$ are pairs $(X, (u,v))$ where $X$ is a group, $u:
H \rightarrow X$ is a morphism of groups, $v: G \rightarrow X$ is
a map such that the following two compatibility conditions hold:
\begin{equation}\eqlabel{2.1.14}
v(g_1)v(g_2) = u (f(g_1, g_2)) v (g_1g_2), \qquad v(g) u(h) =
u(g\triangleright h) v(g)
\end{equation}
for all $g$, $g_1$, $g_2 \in G$ and $h\in H$. A morphism $ f: (X,
(u,v)) \rightarrow (X', (u',v'))$ in $\cat$ is a morphism of
groups $f: X \rightarrow X'$ such that $f \circ u = u'$ and
$f\circ v = v'$. It can be checked that $\bigl( H \#_{\alpha}^{f}
\, G, \, (i_H, i_G) \bigl)$ is an object of $\cat$, where $i_H$
and $i_G$ are the canonical inclusions $i_H : H \rightarrow H
\#_{\alpha}^{f} \, G$, $i_H (h) = (h, 1)$ and $i_G : G \rightarrow
H \#_{\alpha}^{f} \, G$, $i_G (g) = (1, g)$ for all $h\in H$ and
$g\in G$.

Define the category $\Cat$ as follows: the objects in $\Cat$ are
pairs $(X, (u,v))$ where $X$ is a group, $v: X \rightarrow G$ is a
morphism of groups, $u: X \rightarrow H$ is a map such that the
following compatibility condition holds:
\begin{equation}\eqlabel{2.1.16}
u(xy) = u (x) \bigl( v (x)\triangleright u(y) \bigl) f \bigl(v(x),
v(y) \bigl)
\end{equation}
for all $x$, $y \in X$. A morphism $ f: (X, (u,v)) \rightarrow
(X', (u',v'))$ in $\Cat$ is a morphism of groups $f: X \rightarrow
X'$ such that $u' \circ f = u$ and $v'\circ f = v$. It can be
checked that $\bigl( H \#_{\alpha}^{f} \, G, \, (\pi_H, \pi_G)
\bigl)$ is an object of $\Cat$ where $\pi_H$ and $\pi_G$ are the
canonical projections $\pi_H : H \#_{\alpha}^{f} \, G \rightarrow
H$, $\pi_H (h, g) = h$ and $\pi_G : H \#_{\alpha}^{f} \, G
\rightarrow G$, $\pi_G (h, g) = g$ for all $h\in H$ and $g\in G$.

\begin{theorem}\thlabel{2.1.8}
Let $\Gamma = (H, G, \alpha, f)$ be a normalized crossed system of
groups. Then:
\begin{enumerate}
\item $\bigl( H \#_{\alpha}^{f} \, G, \, (i_H, i_G) \bigl)$ is an
initial object of $\cat$.

\item $\bigl( H \#_{\alpha}^{f} \, G, \, (\pi_H, \pi_G) \bigl)$ is
a final object of $\Cat$.
\end{enumerate}
\end{theorem}

\begin{proof}
1. Let $(X, (u,v))\in \cat$. We have to prove that there exists a
unique morphism of groups $w:  H \#_{\alpha}^{f} \, G \rightarrow
X$ such that the following diagram commutes:
$$
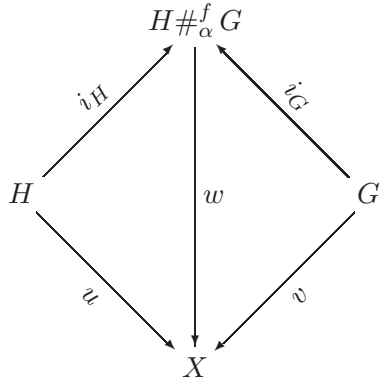
\begin{diagram}
  &          &H \#_{\alpha}^{f} \, G    &          &  \\
  &\NE^{i_H} &               &\NW^{i_G} &  \\
H &          & \dTo_{w}      &          & G\\
  &\SE_{u}   &               &\SW_{v}   &  \\
  &          &  X            &          &
\end{diagram}$$

Assume that $w$ satisfies this condition. Then
\begin{eqnarray*}
w ((h, g))  &=& w ((h, 1) \cdot (1, g))= w ((h, 1))
w((1, g))\\
&=& (w \circ i_H) (h) (w \circ i_G)(g) = u (h) v(g)
\end{eqnarray*}
for all $h\in H$ and $g\in G$ and this proves that $w$ is unique.
The existence of $w$ can be proved as follows: define
$$
w : H \#_{\alpha}^{f} \, G \rightarrow X, \qquad w (h, g) = u(h)
v(g)
$$
Then
\begin{eqnarray*}
&&\hspace*{-2cm} w\bigl( (h_1, g_1)\cdot (h_2, g_2) \bigl ) = w
\bigl ( h_1 (g_1
\triangleright h_2)f(g_1, g_2), \, g_1g_2 \bigl )\\
&=& u(h_1) u(g_1\triangleright h_2) u\bigl(f(g_1, g_2)\bigl) v (g_1g_2)\\
&\stackrel{\equref{2.1.14} } {=}& u(h_1)u(g_1 \rhd h_2)
v(g_1)v(g_2) \\
&\stackrel{\equref{2.1.14} } {=}& u(h_1)v(g_1)u(h_2)v(g_2) =
w(h_1, g_1) w (h_2, g_2)
\end{eqnarray*}
i.e. $w$ is a morphism of groups. The fact that the diagram above
commutes is left to the reader.

2.  Let $(X, (u,v))\in \Cat$. We have to prove that there exists a
unique morphism of groups $w: X \rightarrow H \#_{\alpha}^{f} \, G
$ such that the following diagram commutes:
$$\begin{diagram}
  &          &X  &          &  \\
  &\ldTo^{u} &               &\rdTo^{v} &  \\
H &          & \dTo_{w}      &          & G\\
  &\luTo_{\pi_H}   &               &\ruTo_{\pi_G}   &  \\
  &          &H \#_{\alpha}^{f} \, G            &          &
\end{diagram}$$
Assume that $w$ satisfies this condition; then the commutativity
of the diagrams gives that $w(x) = (u(x), v(x))$ for all $x\in X$,
i.e. $w$ is unique. For the existence of $w$ we define $w : X
\rightarrow H \#_{\alpha}^{f} \, G$, $w(x) = (u(x), v(x))$ for all
$x\in X$. We have:
\begin{eqnarray*}
&&\hspace*{-2cm} w(x) w(y) = \bigl ( u(x), v(x) \bigl)\cdot \bigl
(u(y), v(y)\bigl ) \\
&=& \Bigl( u(x) \bigl(v(x)\triangleright u(y)\bigl) f(v(x), v(y)), \, v(x)v(y) \Bigl)\\
&\stackrel{\equref{2.1.16} } {=}& (u(xy), v(xy)) = w(xy)
\end{eqnarray*}
i.e. $w$ is a morphism of groups, as needed.
\end{proof}

The first application of \thref{2.1.8} is the description of the
(iso)morphism between a crossed product and a group.

\begin{corollary}\colabel{morfisme1}
Let $X$ be a group and $(H, G, \alpha, f)$ be a normalized crossed
system of groups. Then:
\begin{enumerate}
\item A map $w : H \#_{\alpha}^{f} \, G \rightarrow X$ is a
morphism of groups if and only if there exists a pair $(u,v)$,
where $u : H \rightarrow X$ is a morphism of groups, $v : G
\rightarrow X$ is a map such that
\begin{eqnarray}
v(g_1)v(g_2) &=& u (f(g_1, g_2)) v (g_1g_2), \eqlabel{2.1.14'}\\
v(g) u(h) &=& u(g\triangleright h) v(g),  \eqlabel{2.1.14'2}\\
w (h, g) &=& u(h)v(g)  \eqlabel{2.1.14'3}
\end{eqnarray}
for all $g$, $g_1$, $g_2 \in G$ and $h\in H$.

\item A map $\psi : X \rightarrow H \#_{\alpha}^{f} \, G$ is a
morphism of groups if and only if there exists a pair $(u,v)$,
where $u : X \rightarrow H$ is a map, $v : X \rightarrow G$ a
morphism of groups such that
\begin{eqnarray}
u(xy) &=& u (x) \bigl( v(x)\triangleright u(y) \bigl) f
\bigl(v(x), v(y) \bigl) \eqlabel{2.1.16'} \\
\psi (x) &=& \bigl(u(x), v(x) \bigl)  \eqlabel{2.1.16''}
\end{eqnarray}
for all $x$, $y \in X$.

\item $w : H \#_{\alpha}^{f} \, G \rightarrow X$ given by
\equref{2.1.14'3} is an isomorphism of groups if and only if there
exists a pair $(r, s)$, where $r: X\rightarrow G$ is a morphism of
groups and a retraction of $v$ (i.e. $r\circ v = Id_G$), $s:
X\rightarrow H$ is a map that is a retraction of $u$ (i.e. $s\circ
u = Id_H$) such that:
\begin{eqnarray*}
s(x y) &=& s(x)\bigl(r(x)\triangleright
s(y)\bigl)f\bigl(r(x),r(y)\bigl), \quad u\bigl(s(x)\bigl) v\bigl(r(x)\bigl) = x, \\
r(u(h)) &=& 1, \quad s(v(g))= 1
\end{eqnarray*}
for all $x \in X$, $h \in H$ and $g \in G$.

\item $\psi : X \rightarrow H \#_{\alpha}^{f} \, G$ given by
\equref{2.1.16''} is an isomorphism of groups if and only if there
exists a pair $(r,s)$, where $r:H\rightarrow X$ is a morphism of
groups and a section of $u$ (i.e. $u\circ r = Id_H$),
$s:G\rightarrow X$ is a map that is a section of $v$ (i.e. $v\circ
s = Id_G$) such that:
\begin{eqnarray*}
s(g_{1})s(g_{2}) &=&r \bigl(f(g_{1},g_{2})\bigl)s(g_{1}g_{2}),
\qquad s(g)r(h)=r(g\triangleright h)s(g)\\
r\bigl(u(x)\bigl)s\bigl(v(x)\bigl) &=& x, \quad v(r(h)) = 1, \quad
u(s(g))=1
\end{eqnarray*}
for all $g_{1}, g_{2},g \in G$ and $h\in H$.
\end{enumerate}
\end{corollary}

\begin{proof} 1) We shall use \thref{2.1.8}: if
$w : H \#_{\alpha}^{f} \, G \rightarrow X$ is a morphism of groups
then we take $u$ and $v$ given by $u (h) : = w (h, 1)$ and $ v(g):
= w( 1, g)$. Conversely, if $u$ and $v$ are given we define $w (h,
g) := u(h)v(g)$. 2) follows analogous from \thref{2.1.8}.

3) $w: H \#_{\alpha}^{f} \, G \rightarrow X$ is an isomorphism of
groups if and only if there exists $\psi : X \rightarrow H
\#_{\alpha}^{f} \, G$ a morphism of groups such that $\psi \circ w
= Id_{H \#_{\alpha}^{f} \, G}$ and $w \circ \psi = Id_{X}$. Using
$2)$ $\psi$ is a morphism of groups if and only if there exists $r
: X\rightarrow G$ a morphism of groups and a map $s : X\rightarrow
H$ such that:
$$
s(xy) = s(x)\bigl(r(x)\triangleright
s(y)\bigl)f\bigl(r(x),r(y)\bigl), \quad \psi (x) = \bigl( s(x),
r(x)\bigl)
$$
Now if we use $\psi \circ w = Id_{H \#_{\alpha}^{f} \, G}$ and
$w\circ \psi = Id_{X}$ for the generators of $H \#_{\alpha}^{f} \,
G$ the conclusion follows. Item 4) follows similarly.
\end{proof}

In particular we can find the morphisms between two crossed
products. If $\alpha ' : G \to \Aut (H)$ is another weak action we
shall denote $\alpha (g) (h) = g \rhd ' h$.

\begin{corollary}\colabel{morfisme2}
Let $(H, G, \alpha, f)$ and $(H, G, \alpha', f')$ be two
normalized crossed systems. There exists a bijection between the
set of all morphisms of groups $ \psi : H \#_{\alpha}^{f} \, G
\rightarrow H \#_{\alpha'}^{f'}\, G$ and the set of all quadruples
$(u, r, v, s)$, where $u : H \to H$, $r: G\to H$, $ v: G \to G$
are three maps, and $s: H \to G$ is a morphism of groups such that
\begin{eqnarray*}
v (g_1) v (g_2) &=& s \bigl( f (g_1, g_2) \bigl)\, v (g_1 g_2) \\
v (g) s (h) &=& s (g \rhd h) v(g)\\
u (h_1 h_2) &=& u (h_1) \bigl( s(h_1) \rhd ' u (h_2) \bigl) f'
\bigl( s (h_1), s(h_2) \bigl)
\end{eqnarray*}
\begin{eqnarray*}
u \bigl( g \rhd h \bigl) \Bigl( s( g \rhd h) \rhd ' r(g) \Bigl) f'
\Bigl( s( g \rhd h) , \, v(g) \Bigl) &=& r(g) \bigl(v(g) \rhd '
u(h)
\bigl) f' \bigl(v(g), s(h)\bigl) \\
r (g_1) \bigl(v (g_1) \rhd ' r (g_2) \bigl) f' \bigl (v (g_1), v
(g_2)\bigl) &=& u \bigl( f(g_1, g_2)\bigl) \Bigl( s \bigl( f(g_1,
g_2)\bigl) \rhd' r(g_1 g_2) \Bigl)\\
&&f' \bigl(s \bigl( f(g_1, g_2)\bigl), v(g_1g_2) \bigl)
\end{eqnarray*}

Furthermore, the one to one correspondence is given such that
$\psi : H \#_{\alpha}^{f} \, G \rightarrow H \#_{\alpha'}^{f'}\,
G$ is given by the formula
\begin{eqnarray*}
\psi (h, g) &=& \bigl(u(h), s(h)\bigl) \cdot \bigl(r(g), v(g)
\bigl) \\
&=& \Bigl( u(h) \bigl(s(h) \rhd' \, r (g) \bigl) f' \bigl( s(h),
v(g)\bigl), \, s(h)v(g) \Bigl)
\end{eqnarray*}\eqlabel{50a}
and $u(1) = 1$, $v(1) = 1$ and $r(1) = 1$,  for any such quadruple
$(u, r, v, s)$.
\end{corollary}

\begin{proof}
Let $\psi : H \#_{\alpha}^{f} \, G \rightarrow H
\#_{\alpha'}^{f'}\, G$ be a morphism of groups. Using 1) of
\coref{morfisme1} for $X : = H \#_{\alpha'}^{f'}\, G$ we get that
there exists a unique pair $(\Phi_1, \Phi_2)$, where $\Phi_1 : H
\to X$ is a morphism of groups and a map $\Phi_2 : G \to X$ such
that
\begin{eqnarray}
\psi (h, g) &=& \Phi_1 (h) \Phi_2(g) \eqlabel{coc1}\\
\Phi_2 (g_1) \Phi_2 (g_2) &=& \Phi_1 \bigl( f(g_1, g_2) \bigl)
\Phi_2 (g_1g_2)\\
\Phi_2(g) \Phi_1(h) &=& \Phi_1 (g\rhd h)\Phi_2 (g)
\end{eqnarray}
for all $h\in H$, $g$, $g_1$, $g_2 \in G$. $\Phi_1 : H \to H
\#_{\alpha'}^{f'}\, G$ is a morphism of groups: thus using 2) of
\coref{morfisme1} we get a unique pair $(s, u)$, where $s: H \to
G$ is a morphism of groups, $u: H\to H$ is a map such that
\begin{eqnarray}
\Phi_1 (h) &=& (u(h), s(h)) \\
u (h_1 h_2) &=& u(h_1) \bigl(s(h_1) \rhd' u(h_2)\bigl) f' \bigl(
s(h_1), s(h_2) \bigl) \eqlabel{coc3}
\end{eqnarray}
for all $h$, $h_1$, $h_2 \in H$. $\Phi_2 : G \to H
\#_{\alpha'}^{f'}\, G$ is a map: hence there exists a unique pair
of maps  $(r, v)$, where $r: = \pi_H \circ \Phi_2 : G \to H$ and
$v:=\pi_G \circ \Phi_2 : G \to G$ such that $\Phi_2 (g) = (r(g),
v(g))$, for all $g\in G$. To conclude, for any morphism of groups
$ \psi : H \#_{\alpha}^{f} \, G \rightarrow H \#_{\alpha'}^{f'}\,
G$ there exist unique quadruples $(s, u, r, v)$ as above such that
$$
\psi (h, g) = \bigl(u(h), s(h)\bigl) \cdot \bigl(r(g), v(g) \bigl)
= \Bigl( u(h) \bigl(s(h) \rhd' \, r (g) \bigl) f' \bigl( s(h),
v(g)\bigl), \, s(h)v(g) \Bigl)
$$
Moreover, the compatibility conditions \equref{coc1} -
\equref{coc3} are reduced to exactly in the five compatibility
conditions of the statement of the Corollary. Finally, if we
specialize the first condition at $g_1 = g_2 = 1$, we obtain $v(1)
= 1$. Then if we put $h_1 = h_2 = 1$ in the third relation we get
that $u(1) = 1$ and if we let $g=1$ and $h=1$ in the fourth
condition we obtain $r(1) = 1$.
\end{proof}

The morphisms between two crossed products that stabilize the ends
are much easier to describe:

\begin{corollary}\colabel{morfisme23}
Let $(H, G, \alpha, f)$ and $(H, G, \alpha', f')$ be two
normalized crossed systems. There exists a bijection between the
set of all (iso)morphisms of groups $ \psi : H \#_{\alpha}^{f} \,
G \rightarrow H \#_{\alpha'}^{f'}\, G$ such that the diagram
\begin{equation}\eqlabel{diagramacrossed1}
\begin{diagram}
1 \rTo & H & \rTo^{i_H} & H \#_{\alpha}^{f} \, G & \rTo^{\pi_G} & G & \rTo & 1\\
& \dTo_{Id_{H}} & & \dTo_{\psi} & & \dTo_{Id_{G}} & \\
1 \rTo & H & \rTo^{i_H'} & H \#_{\alpha'}^{f'} \, G &
\rTo^{\pi_G'} & G & \rTo & 1
\end{diagram}
\end{equation}
is commutative\footnote{Such a morphism is necessarily an
isomorphism of groups \cite[Theorem 3.2.3]{bechtell}.} and the set
of all maps $r: G\to H$ such that
\begin{eqnarray}
g \rhd' h &=& r(g)^{-1} (g \rhd  h) r(g) \eqlabel{471} \\
f' (g_1, g_2 ) &=&  \bigl(g_1 \rhd ' r (g_2)^{-1} \bigl) r
(g_1)^{-1} f(g_1, g_2) r(g_1 g_2) \eqlabel{501}
\end{eqnarray}
for all $g$, $g_1$, $g_2\in G$ and $h\in H$. Furthermore, the one
to one correspondence is given such that $\psi : H \#_{\alpha}^{f}
\, G \rightarrow H \#_{\alpha'}^{f'}\, G$ is given by the formula
\begin{eqnarray}
\psi (h, g) = \bigl( h r (g), g \bigl)\eqlabel{5010}
\end{eqnarray}
for all $h\in H$ and $g\in G$.
\end{corollary}

\begin{proof}
Let $\psi:H \#_{\alpha}^{f} \, G\rightarrow H \#_{\alpha'}^{f'} \,
G$ be a morphism of groups such that the above diagram is
commutative. Thus there exist three maps $u: H\rightarrow H$, $r:
G\rightarrow H$, $v: G\rightarrow G$ and a morphism of groups $s:
H\rightarrow G$ such that the five relations of \coref{morfisme2}
hold and $\psi\circ i_H = i_H'\circ Id_{H}$, $Id_{G}\circ \pi_G =
\pi_G'\circ \psi$. It follows that $u = Id_{H}$, $v = Id_{G}$ and
$s(h) = 1$ for any $h \in H$. Hence, the five compatibility
conditions of \coref{morfisme2} are reduced to \equref{471} and
\equref{501} and we are done.
\end{proof}

\begin{remark} \coref{morfisme23} is a classification
result: see  \coref{clasclas} below for the exact statement. It
generalizes the Schreier theorem \cite[Theorem 7.34]{rotman} that
can be obtained if we let $H$ to be an abelian group and $ \alpha'
= \alpha$. In this case \equref{471} holds and \equref{501} can be
written as:
$$
f(g_{1},g_{2})f'(g_{1},g_{2})^{-1}=r (g_1) \bigl(g_1 \rhd  r (g_2)
\bigl)r(g_{1},g_{2})^{-1}
$$
for all $g_1$, $g_2 \in G$ which means that $f' f^{-1}$ is
coboundary.
\end{remark}

Using \coref{morfisme1} we can also describe the isomorphisms
between two crossed products: the explicit description is left to
the reader (we refer to \cite{agore} for full details).
\coref{morfisme2} can be used to compute all the (iso)morphism
between all special cases of crossed products. We shall indicate
only two relevant cases: first we shall compute the morphisms
between a semidirect product and a twisted product and then we
shall describe the morphisms between a crossed product and a
direct product.

\begin{corollary}\colabel{semidirectvstwisted}
Let H and G be two groups, $f: G\times G\rightarrow Z(H)$ a
normalized 2-cocycle and $\alpha: G\rightarrow Aut(H)$ a morphism
of groups. There exists a bijection between the set of all
morphisms of groups $\psi: H \ltimes_{\alpha} \, G \rightarrow H
\times^{f} \, G $ and the set of all quadruples $(s,u,r,v)$, where
$s: H\rightarrow G$, $v: G\rightarrow G$ are morphisms of groups
and $u: H\rightarrow H$, $r: G\rightarrow H$ are two maps such
that:
\begin{eqnarray*}
r(g_{1}g_{2}) &=& r(g_{1})r(g_{2})f(v(g_{1}),v(g_{2}))\\
u(h_{1}h_{2}) &=& u(h_{1}) u(h_{2}) f(s(h_{1}),s(h_{2})) \\
v(g)s(h) &=& s(g\rhd h)v(g)\\
r(g)u(h)f(v(g),s(h)) &=& u(g\rhd h)r(g)f(s(g\rhd h),v(g))
\end{eqnarray*}
for all $g$, $g_1$, $g_2\in G$, $h$, $h_1$, $h_2\in H$. Moreover,
through the above bijection $\psi$ is given by
$$
\psi(h,g)=\bigl( u(h) r (g) f(s(h),v(g)), \, s(h)v(g) \bigl)
$$
for all $h\in H$ and $g\in G$.
\end{corollary}

\begin{proof}
We apply \coref{morfisme2} in the case that $f$ is a trivial
cocycle and $\alpha'$ is a trivial action.
\end{proof}

\begin{corollary}\colabel{crossedvsdirect}
Let $(H, G, \alpha, f)$ be a normalized crossed system. There
exists a bijection between the set of all morphisms of groups
$\psi: H \#_{\alpha}^{f}\, G\rightarrow H\times G$ and the set of
all quadruples $(s, u, r, v)$, where $s: H\rightarrow G$, $u:
H\rightarrow H$ are morphisms of groups, and $r: G\rightarrow H$,
$v: G\rightarrow G$ are maps such that:
\begin{eqnarray*}
v(g_{1})v(g_{2}) &=& s(f(g_{1},g_{2}))v(g_{1}g_{2}), \quad  r(g)u(h) = u(g\rhd h)r(g) \\
r(g_{1})r(g_{2}) &=& u(f(g_{1},g_{2}))r(g_{1}g_{2}), \quad
v(g)s(h) = s(g\rhd h)v(g)
\end{eqnarray*}
for all $g$, $g_1$, $g_2\in G$, $h$, $h_1$, $h_2\in H$. Moreover,
through the above bijection $\psi$ is given by
$$
\psi(h,g)=(u(h) r(g), \, s(h)v(g))
$$
for all $g\in G$, $h\in H$.
\end{corollary}

\begin{proof}
We apply \coref{morfisme2} in the case that $f'$ is a trivial
cocycle and $\alpha'$ is a trivial action.
\end{proof}

The next Corollary describes the category of the representations
of a crossed product $H \#_{\alpha}^{f} \, G$.

\begin{corollary}\colabel{reprezentari}
Let $(H, G, \alpha, f)$ be a normalized crossed system and $k$ be
a field. Then there exists an equivalence of categories between
${}_{k [H \#_{\alpha}^{f} \, G]} {\mathcal M}$ the category of
left $k [H \#_{\alpha}^{f} \, G]$-modules and the category of all
triples $(V, \bullet, \star)$ consisting of a left $k[H]$-module
$(V, \bullet)$ and a $k$-linear map $\star : k[G] \otimes_k V \to
V$ such that
$$g_1 \star (g_2 \star v) = f\bigl(g_1, g_2 \bigl)\bullet \bigl(
(g_1g_2) \star v \bigl), \qquad  g\star (h\bullet v) = (g\rhd h)
\bullet (g\star v)$$ for all $g$, $g_1$, $g_2\in G$, $h\in H$ and
$v\in V$.
\end{corollary}

\begin{proof} We apply 1) of \coref{morfisme1} for $X = \Aut_k
(V)$, the group of automorphisms  of a $k$-vector space $V$.
\end{proof}

The next Corollary is also of interest as it reminds us of the
classical Clifford third problem of group representations:

\begin{corollary}\colabel{clifford}
Let $X$ be a group, $(H, G, \alpha, f)$ a normalized crossed
system and $u : H \rightarrow X$ a morphism of groups. Then there
exists a morphism of groups $w: H\#_{\alpha}^{f}\, G \rightarrow
X$ such that the following diagram
$$\begin{diagram}
  & H             &\rTo^{i_H}     & H \#_{\alpha}^{f} \, G \\
  &\dTo^{u}  &\ldTo_{w} &         \\
  & X             &               &
\end{diagram}$$
is commutative if and only if there exists a map $v: G \to X$ such
that
$$
v(g_1)v(g_2) = u (f(g_1, g_2)) v (g_1g_2), \qquad v(g) u(h) =
u(g\triangleright h) v(g)
$$
for all $g_1$, $g_2$, $g \in G$ and $h\in H$.
\end{corollary}

\begin{proof}
We apply 1) of \coref{morfisme1}.
\end{proof}

We shall consider now the exact sequence \equref{sir5} and we can
ask when $i_H : H \to H\#_{\alpha}^{f}\, G $ splits in the
category of groups. For $X: = H$ and $u := {\rm Id}_H$ in the
above corollary we obtain the answer:

\begin{corollary}\colabel{sumand direct}
Let $(H, G, \alpha, f)$ be a normalized crossed system. There
exists a morphism of groups $w: H\#_{\alpha}^{f}\, G \rightarrow
H$ such that $w \circ i_H = {\rm Id}_H$ if and only if there
exists a map $v: G \to H$ such that
$$
g\triangleright h = v(g) h v(g)^{-1}, \qquad f(g_1, g_2) = v(g_1)
v(g_2) v (g_1g_2) ^{-1}
$$
for all $g_1$, $g_2$, $g \in G$ and $h\in H$.
\end{corollary}

Dual to \coref{clifford} we have:

\begin{corollary}
Let $(H, G, \alpha, f)$ be a normalized crossed system, $X$ be a group
and $v : X \rightarrow G$ a morphism of groups. Then there
exists $w: X \rightarrow H\#_{\alpha}^{f}\, G$ a morphism of
groups such that the following diagram
$$
\begin{diagram}
  &                        &                      & X      &   &\\
  &                        &\ldTo^{w}  &\dTo_{v} &   & \\
  & H \#_{\alpha}^{f} \, G   & \rTo^{\pi_G}      & G   &\rTo  &1
\end{diagram}
$$
is commutative if and only if there exists a map $u: X\to H$
such that
$$
u(xy) = u (x) \bigl( v(x)\triangleright u(y) \bigl) f \bigl(v(x),
v(y) \bigl)
$$
for all $x$, $y\in X$.
\end{corollary}

\begin{proof}
We apply the 2) of \coref{morfisme1}.
\end{proof}

\section{Categorical approach: the extension problem revised}\selabel{3}
Let $H$ and $G$ be two groups. We shall define three categories
associated to $H$ and $G$ having the same class of objects, namely
the set ${\rm Crossed} \, (H, G)$ of all normalized crossed
systems. We denote with $\mathcal{E}_1 (H, G)$, $\mathcal{E}_2 (H,
G)$, $\mathcal{E}_3 (H, G)$ the categories having as objects all
normalized crossed systems $(H, G, \alpha, f)$ and morphisms
defined as follows:

$\bullet$ A morphism $\psi : (\alpha, f) \to (\alpha', f')$ in
$\mathcal{E}_1 (H, G)$ is a morphism of groups $ \psi : H
\#_{\alpha}^{f} \, G \rightarrow H \#_{\alpha'}^{f'}\, G$ such
that the diagram
$$
\begin{diagram}
1 \rTo & H & \rTo^{i_H} & H \#_{\alpha}^{f} \, G & \rTo^{\pi_G} & G & \rTo & 1\\
& \dTo_{Id_{H}} & & \dTo_{\psi} & & \dTo_{Id_{G}} & \\
1 \rTo & H & \rTo^{i_H'} & H \#_{\alpha'}^{f'} \, G &
\rTo^{\pi_G'} & G & \rTo & 1
\end{diagram}
$$ is commutative.

$\bullet$ A morphism $(\eta, \psi, \gamma) : (\alpha, f) \to
(\alpha', f')$ in $\mathcal{E}_2 (H, G)$ is a triple $(\eta, \psi,
\gamma)$,  where $\eta : H\to H$, $\gamma: G\to G$ and $ \psi : H
\#_{\alpha}^{f} \, G \rightarrow H \#_{\alpha'}^{f'}\, G$ are
morphisms of groups such that the diagram
\begin{equation}\eqlabel{diamore2}
\begin{diagram}
1 \rTo & H & \rTo^{i_H} & H \#_{\alpha}^{f} \, G & \rTo^{\pi_G} & G & \rTo & 1\\
& \dTo_{\eta} & & \dTo_{\psi} & & \dTo_{\gamma} & \\
1 \rTo & H & \rTo^{i_H'} & H \#_{\alpha'}^{f'} \, G &
\rTo^{\pi_G'} & G & \rTo & 1
\end{diagram}
\end{equation}
is commutative.

$\bullet$ A morphism $\psi : (\alpha, f) \to (\alpha', f')$ in
$\mathcal{E}_3 (H, G)$ is a morphism of groups $ \psi : H
\#_{\alpha}^{f} \, G \rightarrow H \#_{\alpha'}^{f'}\, G$.

\begin{remarks}
1) The category $\mathcal{E}_1 (H, G)$ is a groupoid. This
category was used in the extension problem for groups. The
category $\mathcal{E}_2 (H, G)$ is a diagram-type category and,
from a categorical point of view, defining the morphism as we have
done in $\mathcal{E}_2 (H, G)$ is more natural than the one in
$\mathcal{E}_1 (H, G)$. Finally, in order to classify all crossed
product structures $H \#_{\alpha}^{f} \, G$ that can be
constructed for two fixed groups $H$ and $G$ (i.e. in order to
solve the classification part of Problem 1) we have to deal with
the morphisms as we defined them in $\mathcal{E}_3 (H, G)$.

2) $\mathcal{E}_1 (H, G)$ is a subcategory of $\mathcal{E}_3 (H,
G)$ and there exist functors connecting the above categories as
follows:
$$
F_{23}: \mathcal{E}_2 (H, G) \to \mathcal{E}_3 (H, G), \quad
F_{23} \bigl( (\alpha, f) \bigl) := (\alpha, f), \quad F_{23}
(\eta, \psi, \gamma): = \psi
$$
$$
F_{12}: \mathcal{E}_1 (H, G) \to \mathcal{E}_2 (H, G), \quad
F_{12} \bigl( \alpha, f \bigl) := (\alpha, f), \quad F_{12} (\psi)
:= (Id_H, \psi, Id_G)
$$
\end{remarks}

Having in mind that the forgetful type functors usually have
adjoints we can ask:

\emph{\textbf{Problem 5}: Let $H$ and $G$ be groups. Do the above
functors $F_{12}$, $F_{23}$ or the inclusion functor $i:
\mathcal{E}_1 (H, G)\to \mathcal{E}_3 (H, G)$ have right or left
adjoints? }

The classification part of the extending problem can be restated
as follows: describe the skeleton of the category $\mathcal{E}_1
(H, G)$. The following is of course natural and more general:

\emph{Let $H$ and $G$ be groups. Describe the skeleton of the
categories $\mathcal{E}_i (H, G)$, $i = 1$, $2$, $3$.}

The skeleton of the category $\mathcal{E}_1(H, G)$ is obtained
from \coref{morfisme23} as follows \footnote{In \cite{baez} it is
stated that the skeleton of the category $\mathcal{E}_1 (H, G)$ is
classified by weak $2$-functors $G \to AUT (H)$.} :

\begin{definition} Two normalized crossed systems $(H, G,
\alpha, f)$, $(H, G, \alpha', f') $ are called $1$-equivalently
and we denote it by $(H, G, \alpha,f) \approx_1 (H, G, \alpha',
f')$ if there exists a map $r: G\to H$ such that
\begin{eqnarray*}
g \rhd' h &=& r(g)^{-1} (g \rhd  h) r(g) \\
f' (g_1, g_2 ) &=& \bigl(g_1 \rhd ' r (g_2)^{-1} \bigl) r
(g_1)^{-1} f(g_1, g_2) r(g_1 g_2)
\end{eqnarray*}
for all $g$, $g_1$, $g_2\in G$ and $h\in H$.
\end{definition}

\coref{morfisme23} shows that $(H, G, \alpha,f) \approx_1 (H, G,
\alpha', f')$ if and only if there exists $\psi : (\alpha, f) \to
(\alpha', f')$ an isomorphism in $\mathcal{E}_1 (H, G)$. Thus
$\approx_1$ is an equivalence relation on the set of all
normalized crossed systems ${\rm Crossed}\, (H, G) $ and we have
proved that:

\begin{corollary} \colabel{clasclas}
Let $H$ and $G$ be two groups. There exists a bijection between
the set of objects of the skeleton of the category $\mathcal{E}_1
(H, G)$ and the quotient set ${\rm Crossed} \, (H, G) /\approx_1$.
\end{corollary}

Now we shall describe the skeleton of the category $\mathcal{E}_2
(H, G)$. First we need the following:

\begin{proposition}
Let $H$, $G$ be two groups. Then $(\eta, \psi, \gamma) : (\alpha,
f) \to (\alpha', f')$ is a morphism of $\mathcal{E}_2 (H, G)$ if
and only if $\eta : H\to H$, $\gamma: G\to G$ are morphisms of
groups and there exists a unique map $r: G\rightarrow H$ such
that:
\begin{eqnarray}
\eta \bigl( g \rhd h \bigl) r(g) &=& r(g) \bigl(\gamma(g) \rhd'
\eta(h) \bigl) \eqlabel{gen1} \\
r (g_1) \bigl( \gamma(g_1) \rhd'  r (g_2) \bigl) f' \bigl (\gamma
(g_1), \gamma (g_2)\bigl) &=& \eta \bigl( f(g_1, g_2)\bigl)r(g_1
g_2) \eqlabel{gen2} \\
\psi (h,g) &=& \Bigl( \eta(h) r(g), \gamma(g) \Bigl)
\eqlabel{gen291}
\end{eqnarray}
for all $h\in H$, $g$, $g_1$, $g_2\in G$.
\end{proposition}

\begin{proof}
Since $\psi: H \#_{\alpha}^{f} \, G\rightarrow H \#_{\alpha'}^{f'}
\, G$ is a morphism of groups, from \coref{morfisme2} there exists
a unique quadruple $(u, r, v, s)$, where $u : H \to H$, $r: G\to
H$, $ v: G \to G$ are three maps, and $s: H \to G$ is a morphism
of groups satisfying the compatibility conditions from
\coref{morfisme2}. On the other hand, the diagram
\equref{diamore2} is commutative: thus $\psi(h,1) = (\eta(h), 1) =
(u(h), s(h))$ and $\gamma(g) = s(h)v(g)$. It follows that $u =
\eta$, $v = \gamma$ and $s(h) = 1$ for any $h \in H$. Now, the
first three compatibility conditions of \coref{morfisme2} are
equivalent to $\gamma $ and $\eta$ being morphisms of groups (as
$u$ and $v$ are morphisms and $s(h) = 1$ for all $h$), and the
last two compatibility conditions of \coref{morfisme2} are reduced
to the conditions \equref{gen1}-\equref{gen2}.
\end{proof}

\begin{corollary}\colabel{schclas}
Let $H$ and $G$ be two groups. Then $(\eta, \psi, \gamma) : (
\alpha, f) \to (\alpha', f')$ is an isomorphism in $\mathcal{E}_2
(H, G)$ if and only if $\eta : H\to H$, $\gamma: G\to G$ are
isomorphisms of groups and there exists a unique map $t:
G\rightarrow H$ such that:
\begin{eqnarray}
g \rhd' h &=& \eta \Bigl( t(g) \, \bigl(\gamma^{-1} (g) \rhd
\eta^{-1} (h) \bigl) \, t(g)^{-1} \Bigl)   \eqlabel{gen1a} \\
f' (g_1, g_2) &=& \eta \Bigl( t(g_1) \, \bigl(\gamma^{-1} (g_1)
\rhd t(g_2) \bigl) \, f \bigl(\gamma^{-1} (g_1), \gamma^{-1} (g_2)
\bigl )\, t(g_1g_2)^{-1} \Bigl)  \eqlabel{gen1b}
\end{eqnarray}
for all $h\in H$, $g$, $g_1$, $g_2\in G$. Moreover, the
isomorphism $\psi : H \#_{\alpha}^{f} \, G\rightarrow H
\#_{\alpha'}^{f'} \, G$ is given by the formula:
\begin{eqnarray*}
\psi(h, g) &=& \Bigl( \eta \bigl(h\, t(\gamma (g))^{-1} \bigl ),
\gamma(g) \Bigl)
\end{eqnarray*}
for all $h\in H$ and $g\in G$.
\end{corollary}

\begin{proof}
$\psi : H \#_{\alpha}^{f} \, G\rightarrow H \#_{\alpha'}^{f'} \,
G$ given by \equref{gen291} is an isomorphism if and only if it is
bijective. A morphism of groups $\psi^{-1} : H \#_{\alpha'}^{f'}\,
G \to H \#_{\alpha}^{f}\, G$ that makes the diagram
\equref{diamore2} commutative has the form
$$
\psi^{-1} (h, g) = \bigl( \eta^{-1} (h) t (g), \gamma^{-1}(g)
\bigl)
$$
for a unique map $t : G \to H$ that satisfies two compatibility
conditions similar to \equref{gen1}, \equref{gen2}. Now, if we
write $\psi \circ \psi^{-1} = {\rm Id}_{H \#_{\alpha'}^{f'} \, G}$
and $\psi^{-1} \circ \psi = {\rm Id}_{H \#_{\alpha}^{f} \, G}$ on
the set of generators $(h,1)$ and $(1,g)$ we obtain the relations
$$
\eta^{-1} \bigl( r(g) \bigl ) t \bigl(\gamma(g) \bigl) = 1, \quad
\eta \bigl ( t(g) \bigl) r (\gamma^{-1} (g) ) = 1
$$
or equivalently
$$
r(g) = \eta \bigl( t (\gamma (g))^{-1} \bigl)
$$
for all $g\in G$. With this $r$ the compatibility conditions
\equref{gen1}, \equref{gen2} give exactly  \equref{gen1a},
\equref{gen1b}.
\end{proof}

\begin{remark}
In particular, if we specialize  \coref{schclas} for $(H, G,
\alpha, f)$, the trivial crossed system, we obtain a necessary and
sufficient condition for a crossed product to be isomorphic to a
direct product in the category $\mathcal{E}_2 (H, G)$. More
precisely, $\psi: H\times G \to H \#_{\alpha'}^{f'} \, G$ is an
isomorphism in $\mathcal{E}_2 (H, G)$  if and only if there exists
a pair $(\eta, t)$, where $\eta \in \Aut (H)$ is an automorphism
of $H$, $t: G\rightarrow H$ is a map such that:
\begin{eqnarray*}
g \rhd' h &=& \eta \Bigl( t(g) \,
\eta^{-1} (h)  \, t(g)^{-1} \Bigl)  \\
f' (g_1, g_2) &=& \eta \Bigl( t(g_1) t(g_2) t(g_1g_2)^{-1} \Bigl)
\end{eqnarray*}
for all $h\in H$, $g$, $g_1$, $g_2\in G$.
\end{remark}

\begin{definition} Two normalized crossed systems $(H, G,
\alpha, f)$, $(H, G, \alpha', f')$ are called $2$-equivalently and
we denote it by $(H, G, \alpha,f) \approx_2 (H, G, \alpha', f')$
if there exists a triple $(\eta, \gamma, t)$, where $\eta :
H\rightarrow H$, $\gamma : G\rightarrow G$ are isomorphisms of
groups, $t : G\rightarrow H$ is a map such that the compatibility
conditions \equref{gen1a}, \equref{gen1b} hold.
\end{definition}

\coref{schclas} shows that $(H, G, \alpha,f) \approx_2 (H, G,
\alpha', f')$ if and only if there exists $(\eta, \psi, \gamma) :
(\alpha, f) \to (\alpha', f')$ an isomorphism in $\mathcal{E}_2
(H, G)$. Thus $\approx_2$ is an equivalence relation on the set of
all crossed systems ${\rm Crossed}\, (H, G) $ and we have proved
the following classification result that is a general Schreier's
type theorem:

\begin{theorem} \thlabel{teoprinci}
Let $H$ and $G$ be two groups. There exists a bijection between
the set of objects of the skeleton of the category $\mathcal{E}_2
(H, G)$ and the quotient set ${\rm Crossed} \, (H, G) /\approx_2$.
\end{theorem}

We end the paper with the following:

\emph{\textbf{Problem 6}: Construct the crossed product for
groupoids and generalize the results presented in this paper to
the level of groupoids.}

We recall that a groupoid is a small category in which any
morphism is an isomorphism. This way, groups are groupoids with
only one object. The construction of the crossed product for
groupoids must be made in such a way that its generalization to
the level of Hopf algebroids agrees with the one recently
constructed in \cite{BB}.

\textbf{Acknowledgement:} We thank the referees for their detailed
suggestions that have helped us improve this paper.

\end{document}